\documentclass[11pt,a4paper
]{amsart}
\usepackage[utf8]{inputenc}
\usepackage[english
]{babel}
\usepackage{amsmath}
\usepackage{amsfonts}
\usepackage{amssymb}
\usepackage{amsthm}
\usepackage{tikz-cd}
\usepackage{float}

\usepackage[colorlinks=true]{hyperref}
\usepackage{graphicx}
\usepackage{latexsym}    
\usepackage{amssymb}   
\usepackage{amsmath}    
\usepackage{amsbsy}
\usepackage{amsthm}
\usepackage{amsgen}
\usepackage{amsfonts}
\usepackage{array}
\usepackage{lmodern}
\usepackage[all]{xy}    

\usepackage{graphicx}

\usepackage{caption}
\usepackage[labelformat=simple]{subcaption}

\newsavebox{\largestimage}

\usepackage{tikz}
\usetikzlibrary{knots}
\usetikzlibrary{decorations.pathmorphing}
\usetikzlibrary{calc, decorations.markings}

\usepackage{color}
\usepackage{verbatim}
\usepackage{url}
\usepackage{enumitem}
\numberwithin{figure}{section}

\tikzset{snake it/.style={decorate, decoration=snake}}

\usepackage{lineno}

\usepackage{comment}

\usepackage[colorinlistoftodos, textwidth=3cm, shadow]{todonotes}
\newcounter{jkcomment}

\newcounter{dccomment}

\usepackage{marginnote}
\makeatletter
\renewcommand{\@todonotes@drawMarginNoteWithLine}{%
  \begin{tikzpicture}[remember picture, overlay, baseline=-0.75ex]%
    \node [coordinate] (inText) {};%
  \end{tikzpicture}%
  \marginnote[{
    \@todonotes@drawMarginNote%
    \@todonotes@drawLineToLeftMargin%
  }]{
    \@todonotes@drawMarginNote%
    \@todonotes@drawLineToRightMargin%
  }%
}
\makeatother

\newcommand{\blank}{{-}} 

\DeclareMathOperator{\Diff}{Diff}

\DeclareMathOperator{\dive}{div}
\DeclareMathOperator{\expo}{exp}

\DeclareMathOperator{\Iso}{Isom}
\DeclareMathOperator{\kernel}{Ker}
\DeclareMathOperator{\Mod}{\mathcal{M}}

\DeclareMathOperator{\Riem}{\mathcal{R}}

\DeclareMathOperator{\vol}{vol}

\makeatletter
\providecommand*{\dop}%
{\@ifnextchar^{\@diffOperator}{\@diffOperator^{}}}
\def\@diffOperator^#1{%
  \mathop{\mathrm{\mathstrut d}}%
  \nolimits^{#1}\gobblespace}

\providecommand*{\Dop}%
{\@ifnextchar^{\@DiffOperator}{\@DiffOperator^{}}}
\def\@DiffOperator^#1{%
  \mathop{\mathrm{\mathstrut D}}%
  \nolimits^{#1}\gobblespace}

\def\gobblespace{%
  \futurelet\diffarg\opspace}
\def\opspace{%
  \let\DiffSpace\!%
  \ifx\diffarg(%
  \let\DiffSpace\relax
  \else
  \ifx\diffarg[%
  \let\DiffSpace\relax
  \else
  \ifx\diffarg\{%
  \let\DiffSpace\relax
  \fi\fi\fi\DiffSpace}
\makeatother

\newtheorem{thm}{Theorem}[section]

\newtheorem{lemma}[thm]{Lemma}
\newtheorem{cor}[thm]{Corollary}

\newtheorem{proposition}[thm]{Proposition}
\theoremstyle{definition}	
\newtheorem{remark}[thm]{Remark}
\newtheorem{dfn}[thm]{Definition}
\newtheorem*{conv}{Conventions}
\newtheorem*{ack}{Acknowledgements}

\newtheoremstyle{TheoremNoNum}
{.5\baselineskip}{.5\baselineskip}              
{\itshape}                      
{}                              
{\bfseries}                     
{.}                             
{ }                             
{\thmname{#1}\thmnote{ \bfseries #3}}
\theoremstyle{TheoremNoNum}

\usepackage{aliascnt}
\newaliascnt{slicethm}{theorem}
\makeatletter
\newtheorem{slicethm}[slicethm]{Slice Theorem}
\aliascntresetthe{slicethm}
\providecommand*{\slicethmautorefname}{Slice Theorem\@gobble}
\newaliascnt{tubneighthm}{theorem}
\newtheorem{tubneighthm}[tubneighthm]{Theorem B}
\aliascntresetthe{tubneighthm}
\providecommand*{\tubneighthmautorefname}{Theorem B\@gobble}
\makeatother

\newtheoremstyle{TheoremNum}
{\topsep}{\topsep}              
{\itshape}                      
{}                              
{\bfseries}                     
{.}                             
{}                             
{\thmname{#1}\thmnote{ \bfseries #3}}
\theoremstyle{TheoremNum}

\numberwithin{equation}{section}

\usepackage{tikz}

\usepackage{hyperref}
\hypersetup{
  colorlinks=true,
  linkcolor=blue,
  citecolor=blue,
  urlcolor=blue,
  pdfauthor={Diego Corro-Jan-Bernhard Kordass},
  pdftitle={}
}

\newcommand{\N}{\mathbb{N}}	 
\newcommand{\R}{\mathbb{R}} 

\usepackage{lipsum}

\usepackage{physics}

\usepackage{ifdraft}

\author[D.~Corro]{Diego Corro$^*$}
\address[D.~CORRO]{Institut f\"ur Algebra und Geometrie, Karlsruher Institut f\"ur Technologie (KIT), Karlsruhe, Germany.}
\email{\href{mailto:diego.corro@partner.kit.edu}
  {diego.corro@partner.kit.edu}}
\urladdr{\url{http://www.math.kit.edu/iag5/~corro/en}}
 \thanks{$^*$Supported by the DFG (281869850, RTG 2229 ``Asymptotic Invariants and Limits of Groups and Spaces").}

\author[J.-B.~Korda\ss]{Jan-Bernhard Korda\ss$^*$}
\address[J.-B.~KORDA\ss]{Institut f\"ur Algebra und Geometrie, Karlsruher Institut f\"ur Technologie (KIT), Karlsruhe, Germany.}
\curraddr{Département de mathématiques, Université de Fribourg, Switzerland.}
\email{\href{mailto:kordass@kit.edu}
  {kordass@kit.edu}}
\urladdr{\url{http://www.math.kit.edu/iag5/~kordass/en}}

\title[EXISTENCE SLICES SPACE OF RIEMANNIAN METRICS]{SHORT SURVEY ON THE EXISTENCE OF SLICES FOR THE SPACE OF RIEMANNIAN METRICS}
\date{\today}


\begin{document}
\overfullrule=5pt
\begin{abstract}
  We review the well-known slice theorem of Ebin for the action of the diffeomorphism group on the space of Riemannian metrics of a closed manifold.
  We present advances in the study of the spaces of Riemannian metrics, and produce a more concise proof for the existence of slices.
\end{abstract}

\maketitle


\section{Introduction}



The presence of a group action on a smooth manifold has been a successful tool in the study of the topology of manifolds, as well as their geometry (cf.\ \cite{AlexandrinoBettiol,Bredon,Kobayashi}).
In \cite{Palais1961}, Palais showed that given certain topological conditions on the action of the group, one could integrate the normal bundle of an orbit to obtain an equivariant tubular neighborhood called a \emph{slice}.
A result guaranteeing its existence is referred to as a \emph{slice theorem}.
Such a theorem allows us to study the action of a Lie group $G$ on $M$ locally, as well as the quotient map $q \colon M \to M/G$.
The required topological condition is automatically satisfied by any compact group, and slices have played an important role in the study of such actions.

Furthermore, the proof of the slice theorem, for the case when the group $G$ and the manifold $M$ are finite-dimensional, is rather geometric.
It has been exploited in the so called \emph{Symmetry Program}, to study and construct, metrics satisfying lower curvature bounds on a given smooth manifold, with a prescribed group contained in the symmetry group (see \cite{Grove2000, Grove2018}).

When attempting to reproduce this proof for the case when the dimensions of the acting group $\mathcal G$ and the manifold $\mathcal M$ are both infinite, a series of technical obstructions arise.
Each of these technical points has to be addressed in order to show the existence of a slice.
This is the core of Ebin's work in \cite{Ebin1968} for  the action of the group of diffeomorphisms $\Diff(M)$ on the space of Riemannian metrics $\Riem(M)$ for a closed smooth manifold $M$.
The action is by pulling back Riemannian metrics along diffeomorphisms.
To view it as a left action, we set $\phi \cdot g := (\phi^{-1})^{*}g$ for $\phi \in \Diff(M)$ and $g \in \mathcal R(M)$.
Namely, in \cite{Ebin1968} the following theorem is proved.

\begin{slicethm}\label{thm:slice-thm}
  Let $M$ be a closed smooth manifold and consider the Fré\-chet Lie group of diffeomorphisms $\Diff(M)$.
  Given a Riemannian metric $\gamma \in \Riem(M)$, there exists a submanifold $S_\gamma$ containing $\gamma$ such that the following hold:
  \begin{enumerate}[label=(\roman*)]
  \item For any $\phi \in \Diff(M)_\gamma$, we have $\phi \cdot S_\gamma \subset S_\gamma$.
  \item If $\phi \in \Diff(M)$ is a diffeomorphism such that $\phi \cdot S_\gamma \cap S_\gamma\neq \emptyset$, then $\phi \in \Diff(M)_\gamma$.
  \item There exists an open neighborhood $U$ of the identity right-coset in $\Diff(M)/\Diff(M)_\gamma$, and a cross-section $\chi\colon U\to \Diff(M)$, such that the map $F\colon U\times S_\gamma\to \Riem(M)$ given by 
    \[
      F(u,s) = \chi(u) \cdot s,
    \]
    is a homeomorphism onto an open neighborhood of $\gamma$ in $\mathcal R(M)$.
  \end{enumerate}
\end{slicethm}

The slice theorem has acquired recent relevance in Riemannian geometry regarding the study of the \emph{moduli space} $\Mod(M)$ of Riemannian structures on $M$.
This space is the orbit space $\Riem(M)/\Diff(M)$, and parametrizes possible Riemannian geometries on $M$ up to isometry.
Since the proof of the~\autoref{thm:slice-thm} is by local means, the slice theorem holds for any open $\Diff(M)$-invariant submanifold of $\Riem(M)$.
Particular examples of such invariant open manifolds are the subsets of Riemannian metrics satisfying some strict lower (resp.\ upper) curvature bound.
For example, we may consider the space of Riemannian metrics with positive sectional curvature $\Riem^{\sec >0}(M)$.
The orbit space $\Mod^{\sec >0}(M) := \Riem^{\sec >0}(M)/\Diff(M)$ describes isometry classes of metrics with positive sectional curvature on $M$.
Moreover, we one can consider the action of closed subgroups of $\Diff(M)$, such as $\Diff^0(M)$, the subgroup of diffeomorphism homotopic to the identity.
A particular example for a space of metrics with an upper curvature bound is the Teichmüller space (cf.\ \cite{FarrellSocar2016}).


Ebin's approach to prove the~\autoref{thm:slice-thm} is as follows.
He first considers a fixed Riemannian metric and a fixed volume element on $M$.
Then he fixes a degree $s \in \N \cup \{ \infty \}$ of differentiability, and defines an inner product $(\blank,\blank)_s$ on the space of sections $\Gamma^s (S^2 T^\ast M)$.
The downside is that the space $\Gamma^s (S^2 T^\ast M)$ is not complete with respect to this inner product.
Completing it yields a Sobolev space $H^s(S^2 T^\ast M)$.
Furthermore, the topology of this Sobolev space does not depend on the choices of the Riemannian metric, and the volume element of $M$ (see \cite[Sec.~3]{Ebin1968}).
Ebin then proceeds to show the existence of a slice for the action of the group of diffeomorphisms on this Sobolev space. 


For $M$ a closed smooth manifold, the study of a Riemannian structure for the space $\Riem(M)$ has greatly advanced in the last 50 years.
Namely, Freed and Groisser in \cite{FreedGroisser1989}, and later  Gil-Medrano and Michor in \cite{Gil-MedranoMichor1991} have studied the space $\Riem (M)$ with a canonical non-complete Riemannian structure, given by an $L^2$-metric (cf.\ \cite{KrieglMichor}).
They study the Levi-Civita connection of such a metric, and show the existence of an exponential map.

Using these advances we present an alternative proof of the~\autoref{thm:slice-thm}. 

\begin{remark}
The original result in \cite{Ebin1968} considers only orientable manifolds. In the present proof we believe this hypothesis is not necessary.
\end{remark}

The proof we present does not rely on the technical work done in the setting of Sobolev spaces in \cite{Ebin1968}.
Furthermore, the advances in the study  of topological group actions gives a clearer picture of the construction used in the proof of the slice.
This allows to describe up to homeomorphism a neighborhood of an orbit of the action of $\Diff(M)$ on $\Riem(M)$.
We observe that the isotropy group of a Riemannian metric $\gamma \in \Riem(M)$ is the group of isometries $\Iso(\gamma)$ of $\gamma$.
When $M$ is compact, the group $\Iso(\gamma)$ is a compact Lie group (see \cite{MeyersSteenrod1939}).

\begin{tubneighthm}\label{thm:tubular-nbhd-thm}
  Let $M$ be a closed smooth manifold, and consider $\Diff(M)$ acting on $\Riem(M)$.
  For a fixed Riemannian metric $\gamma$ let $S_\gamma$ be the slice through $\gamma$.
  Then a closed neighborhood of the orbit $\Diff(M) (\gamma) \subset \Riem(M)$ is homeomorphic to 
  \[
    \Diff(M) \times_{\Iso(\gamma)} S_\gamma
  \]
\end{tubneighthm}

This theorem allows to describe a neighborhood of the orbits, which can be regarded as a topological equivariant tubular neighborhood.

\begin{remark}
The original  approach followed by Ebin relies on the fact that the Fréchet manifold of Riemannian manifolds has a graded family of Riemannian metrics generating the topology. This approach has been generalized in \cite{DiezRudolph2019}, to show the existence of slices for a more general family of Fréchet spaces and Fréchet Lie group actions.
\end{remark}

This article is organized as follows.
In the first part we present the general theory of smooth actions by finite-dimensional Lie groups on finite-dimensional manifolds.
For this setting we give a proof of a slice theorem, to which we will point later when proving a slice theorem for the action of $\Diff(M)$ on $\Riem(M)$.
We recall an infinite-dimensional manifold structure for $\Riem(M)$, which is modeled on a class of metric spaces called Fr\'{e}chet spaces.
Then we present the $L^2$-metric on $\Riem(M)$ studied by Gil-Medrano and Michor in \cite{Gil-MedranoMichor1991}.
This $L^2$-metric is $\Diff(M)$ invariant, and has an exponential map.
Finally, we present the proof of the~\autoref{thm:slice-thm} and the~\autoref{thm:tubular-nbhd-thm}.
These are based on the proofs for the finite-dimensional case.
We end this work by stating some consequences of the slice theorem.

\vspace*{2ex}
\begin{ack}
The authors would like to thank the referee for his helpful comments.
\end{ack}

\section{Preliminaries}

\begin{conv}
  We will henceforth use the following conventions:
  \begin{enumerate}
  \item All manifolds, submanifolds and Lie groups are assumed to be \linebreak smooth and finite-dimen\-sional, if not stated otherwise.
  \item By $M$ we always denote a closed, finite-dimensional smooth manifold.
  \item We denote by $G$ a Hausdorff topological group; often a finite-dimen\-sional Lie group.
  \item The continuous map describing an action of a group on a space is denoted by $\mu$.
  \item For a smooth vector bundle $E \to M$, we denote by $\Gamma^s(E)$ for $s \in \N \cup \{\infty\}$ the space of sections of differentiability $s$.
    Instead of $\Gamma^{\infty}(E)$, we will write $\Gamma(E)$.
  \end{enumerate}
\end{conv}

\subsection{Group actions}

A group $G$ equipped with a topology is called a \emph{topological group}, if the multiplication $m \colon G \times G \to G$, and the inverse $i \colon G \to G$ are continuous maps with respect to this topology.
In the present work we will additionally assume that all topological groups are Hausdorff.
This implies that for the identity element $e \in G$, the set $\{e\}$ is closed.
The condition of $\{e\}$ being closed is equivalent to $G$ being metrizable, and the metric can be chosen left-invariant by \cite[Sec.~1.22]{MontgomeryZippin}.

Let $G$ be a topological group, and $X$ be a topological space.
We say that \emph{$G$ acts on $X$} from the left, if there exists a continuous map $\mu \colon G \times X \to X$ such that:
\begin{enumerate}[label=(\roman*)]
\item\label{item:grp-action-i} $\mu (e, p) = p$ for any $p \in X$;
\item\label{item:grp-action-ii} $\mu(g, \mu(h,p)) = \mu(gh,p)$ for any $g,h\in G$ and any $p \in X$. 
\end{enumerate}

For a fixed $p \in X$ the set $G(p) := \{ \mu(g,p) \mid g \in G\} \subset X$ is called the \emph{orbit of $p$}.
The subgroup $G_p = \{g\in G\mid \mu(g,p) = p\}$ is called the \emph{isotropy subgroup} at $p$, or the \emph{stabilizer} of the action at $p$.
From now on, we will denote $\mu(g,p)$ simply by $g\cdot p$.
The subgroup given by intersection of all isotropy groups 
\[
  \bigcap_{p\in X} G_p,
\]
is called the \emph{ineffective kernel} of the action.
When this subgroup is trivial, the action is called \emph{effective}.
Any group action can be turned into an effective action and hence we will assume that all group actions in the present text are effective.
Moreover, if at any $p\in X$ we have $G_p=\{e\}$, we say the action is \emph{free}.
An action is \emph{transitive} if the orbit of any point $p\in X$ is the whole space $X$.

An action of $G$ on $X$ defines an equivalence relation $\sim$ on $X$.
We say $p\sim q$, if and only if, $p$ and $q$ are contained in the same orbit.
The quotient space obtained from the action's orbits is called the \emph{orbit space} of the action, and is denoted by $X/G$.
The topology of $X$ induces via the \emph{orbit quotient map} $\pi\colon X\to X/G$ the quotient topology on $X/G$.
For any subset $A \subset X$, we denote its image under $\pi$ by $A^\ast$.
For a point $p \in X$ we denote its image as $p^\ast$.

An action $\mu$ of a topological group $G$ on a topological space $X$ is called \emph{proper}, if the map $A \colon G \times X \to X \times X$ defined as
\[
  A(g,p) = (g \cdot p, p), 
\]
is proper, i.e.\ the preimage of any compact subset of $X$ under the map is a compact subset of $G \times X$.
This concept was introduced by Palais in \cite{Palais1961}.
If $X$ is metrizable, then we have the following characterization of proper actions.

\begin{proposition}[Prop.~3.19 in \cite{AlexandrinoBettiol}]\label{prop:charact-properness}
  Let $G$ be a Hausdorff topological group and let $X$ be a metrizable topological space.
  An action $\mu \colon  G \times X \to X$ is proper, if and only if, for any sequence $\{g_n\}$ in $G$ and any convergent sequence $\{x_n \}$ in $X$, such that $\{\mu(g_n, x_n)\}$ converges in $X$, the sequence $\{g_n\}$ admits a convergent subsequence.
\end{proposition}

From this characterization of proper actions it easily follows that for a proper action all isotropy groups are closed.

In the particular case when the group $G$ is metrizable via a complete metric,
we have the following characterization for a proper action:

\begin{proposition}\label{prop: second characterization of proper action}
%
Let $G$ be a topological group with a complete left-invariant metric inducing the topology of $G$, and $X$ a metric space. An action $\mu\colon G\times X\to X$  is proper, if and only if, for any sequence $\{g_n\}$ in $G$ and any $x_0\in X$, if the sequence $\{\mu(g_n,x_0)\}$ converges to $x_0$ in $X$, then $\{g_n\}$ has a convergent subsequence.
\end{proposition}

\begin{proof}
It is clear that the conclusion of Proposition~\ref{prop:charact-properness}, implies the necessity direction.

Denote by  $\rho\colon G\times G\to \R$ the complete left-invariant metric of  $G$.

We will show that for any $x\in X$, the isotropy group $G_x$ is compact. I.e.\ that any sequence  $\{g_n\}$ in $G_x$ has a convergent subsequence, with limit in $G_x$. Since we have $\rho$ an invariant metric,  $G_x$ is a metric space.  Consider $\{g_n\}\subset G_x\subset G$ an arbitrary sequence. Then for all elements in the sequence, we have that $\mu(g_n,x) = x$. Thus the sequence $\{\mu(g_n,x)\}$ is constant and thus it is a convergent sequence. From our hypothesis, we conclude that there exists a convergent subsequence $\{g_{n_k}\}$, with limit some $g_0 \in G$. From the continuity of $\mu$, we have that $\mu(g_{n_k},x)=x$ converges to $\mu(g,x)$. By the uniqueness of limits, we see that $\mu(g,x) = x$. Thus $g$ is an element of $G_x$, and we conclude that $G_x$ is compact.

Since $G_x$ is compact for any $x\in X$, we can assume, by averaging $\rho$ over $G_x$, that $\rho$ is $G_x$-invariant, by left multiplication. This induces a complete metric $\tilde{\rho}$ on the quotient space $G/G_x$. 

Next we show that the action is proper by showing that the characterization given in Proposition~\ref{prop:charact-properness} holds. Consider $\{(g_n,x_n)\}$ an arbitrary sequence in  $G\times X$  such that the sequence $\{(\mu(g_n,x_n),x_n)\}$ converges in $X\times X$ to $(y,x)$. We will show that there is a convergent subsequence of $\{g_n\}$. Given $\varepsilon$, for sufficiently  large $n\in \N$, we have that for the left-invariant metric $d$ on $X$, that $d(\mu(g_n,x_n),\mu(g_n,x)) = d(x_n,x) < \varepsilon/2$. Since the sequence $\{\mu(g_n,x_n)\}$ converges to $y$, for sufficiently large $n$ we have
\[
 d(y,\mu(g_n,x)) \leqslant d(y,\mu(g_n,x_n)) + d(\mu(g_n,x_n),\mu(g_n,x)) < \varepsilon.
\]
Therefore, for $n$ and $m$ large enough, 
\[
\varepsilon > d(\mu(g_n,x),\mu(g_m,x))= d(\mu(g_m^{-1}g_n,x),x).
\]
From our hypothesis, there is a convergent subsequence, which we denote again by $g_m^{-1}g_n$ converging to some $g_0\in G_{x}$. This implies, that for the complete metric $\tilde{\rho}$ on $G/G_x$ we have for large indices $n$ and $m$:
\[
	\tilde{\rho}(g_nG_{x},g_mG_{x}) <\varepsilon.
\]
This means, that the sequence $\{g_nG_{x}\}$ in $G/G_{x}$ is a Cauchy sequence. Since $\tilde{\rho}$ is complete, there exists a convergent subsequence of $\{g_nG_{x}\}$ converging to $gG_{x_0}$ for some $g\in G$. I.e.\ for large $n$ we have 
\begin{align*}
	\varepsilon > \tilde{\rho}(g_nG_x,gG_x) &= \inf\{\rho(g_na_1,g a_2)\mid a_1, a_2\in G_x\}\\[0.5em]
	 &= \inf\{\rho(g_na_1a_2^{-1},g)\mid a_1, a_2\in G_x\}\geqslant 0.
\end{align*}
Thus there is a sequence $\{g_nh_n\}$ in $G$ converging to $g$. Since $G_x$ is compact, we can find a subsequence $h_{n_k}$ converging to $h\in G_x$. From this, using the triangle inequality, we obtain that the sequence $g_{n_k}$ converges to $h^{-1}g$ in $G$.
\end{proof}


Furthermore, when $G$ is a smooth manifold, and the group operations $m \colon G \times G\to G$, and $i \colon G\to G$ are smooth we say that $G$ is a \emph{Lie group}.
If $M$ is a fixed smooth manifold, and $G$ is a Lie group, then a \emph{smooth action} by $G$ on $M$ is a smooth map $\mu \colon G \times M \to M$ satisfying \ref{item:grp-action-i} and \ref{item:grp-action-ii} above.

The orbit space of a proper smooth action of a Lie group on a smooth manifold, is a second-countable Hausdorff space.
If the action is not free, the quotient might not be a closed manifold.
For example, if we consider the $2$-sphere and fix an axis of rotation, we have a smooth action by the circle on the $2$-sphere with two fix points.
The orbit space is homeomorphic to a closed interval.

\begin{thm}\label{t: orbit of proper action in homeo to G/Gp}
  Let $G$ be a Hausdorff topological group acting properly on a metric space $(X,d)$, where $d$ is $G$-invariant.\footnote{I.e. $d(g \cdot p, g \cdot q) = d(p,q)$ for all $p,q \in X$ and $g \in G$.}
  Fix $p \in X$, and consider $\mu_p\colon G \to X$ given by $\mu_p(g) = g \cdot p$.
  Let $\rho \colon G \to G/G_p$ be the quotient map.
  Then there exists a $G$-equivariant homeomorphism $\tilde{\mu}_p\colon G/G_p \to G(p)$ onto the orbit through $p$ such that $\tilde{\mu}_p \circ \rho = \mu_p$.
  Furthermore, the orbit $G(p)$ is a closed subspace of $X$.
  \begin{center}
    \begin{tikzcd}[sep=large]
      G\arrow{d}[swap]{\rho}\arrow{dr}{\mu_p} &\\
      G/G_p \arrow{r}[swap]{\tilde{\mu}_p} & X
    \end{tikzcd}
  \end{center}
\end{thm}

\begin{proof}
  Consider $g_1G_p = g_2G_p$.
  Then $g_2^{-1}g_1\in G_p$, and thus $g_2^{-1}g_1\cdot p = p$.
  Therefore, we have $g_1 \cdot p = g_2 \cdot p$.
  I.e if $\rho(g_1) = \rho(g_2)$, then $\mu_p(g_1) = \mu_p(g_2)$.
  Since $\rho$ is a quotient map, there exists a continuous, well-defined map $\tilde{\mu}_p\colon G/G_p\to X$ making the diagram commute.
  This map is injective, and its image is the orbit $G(p)$.
  We only have to proof that $\tilde{\mu}_p$ is closed.
  Observe that since we have an invariant metric, we can talk about convergent sequences.
  Take $C\subset G/G_p$ closed, and consider $\tilde{\mu}_p(C)$ in $X$.
  Consider a sequence $\{g_n\cdot p\}$ in $\tilde{\mu}_p(C)$, with limit $g\cdot p$.
  From the properness of the $G$-action we have that the sequence $\{g_n\}$ converges to $g$ in $G$.
  Since $\rho^{-1}(C)$ is closed in $G$, we have that $g$ lies in $\rho^{-1}(C)$, and thus $g G_p$ is an element of $C$.
\end{proof}

\begin{remark}
\hfill
  \begin{enumerate}
  \item We used the properness of the action in Theorem~\ref{t: orbit of proper action in homeo to G/Gp} to show that the inverse of $\tilde{\mu}_p$ is continuous.
  \item In the following section, we will see that for a smooth, proper, effective group action by a Lie group on a smooth manifold $M$, we can always find an invariant metric on $M$.
  In this case, the conclusions of Theorem~\ref{t: orbit of proper action in homeo to G/Gp} can be strengthened to diffeomorphism for $\tilde{\mu}_p$, and smooth embedded submanifold for $G(p)$ (cf.\ \cite[Proposition~3.41]{AlexandrinoBettiol}).
\item Moreover, we will see that the existence of a slice yields an improved version of this theorem topologically describing a neighborhood of the orbits.
  \end{enumerate}
\end{remark}

\subsection{The finite-dimensional slice theorem}\label{ss: classical slice theorem}
A group action gives a partition of a smooth manifold $M$, whose global structure is described by the topological space $M/G$.
Thus, we can attempt to recover information about $M$ from a separate analysis of the orbits, and the orbit space.
Proper actions, which where introduced by Palais in \cite{Palais1961}, are a good setting for this decomposition study.
This follows from the fact that for a proper smooth action an orbit is an embedded closed submanifold (see \cite[Proposition 3.41]{AlexandrinoBettiol}).
Actually, the orbit at $p\in M$ is diffeomorphic to $G/G_p$, and as stated above the orbit space is a second-countable Hausdorff topological space.

There are stronger geometric consequences to the properness of a smooth action by a Lie group $G$ on a smooth manifold $M$.
We say that a Riemannian metric $\gamma$ on $M$ is $G$-invariant, if for any $g \in G$, the map
\begin{align*}
  \mu_g \colon M \to M,
  \quad
  p \mapsto g \cdot p,
\end{align*}
is an isometry of $\gamma$.
The following theorem states that for any proper smooth action by $G$ on $M$ there exists at least one $G$-invariant metric.

\begin{thm}[Theorem 3.65 in \cite{AlexandrinoBettiol}]\label{t: existence of G-invariant Riemannian metric}
  Let $G$ be a Lie group, and let $M$ be a smooth manifold.
  Further, let $\mu \colon G \times M \to M$ be a proper smooth action.
  There exists a $G$-invariant Riemannian metric $\gamma$ on $M$ such that $\mu_G = \{ \mu_g \mid g \in G\}$ is a closed Lie subgroup of $\Iso (M, \gamma)$.
\end{thm}

We note that $\Iso (M, \gamma)$ is a Lie group by the theorem of Myers-Steenrod (see \cite{MeyersSteenrod1939}).

Combining this theorem with the fact that the orbits are closed, the orbit space obtains the structure of a metric space, and the quotient map $\pi\colon M \to M/G$ is a \emph{submetry} with respect to the metrics on $M$ and $M/G$.
This means that the ball of radius $\delta$ centered at $p$ in $M$ is mapped into the ball of radius $\delta$ centered at $p^\ast$ in $M/G$. 

\begin{remark}
Furthermore by \cite{Kankaanrinta2005}, the $G$-invariant metric can be considered complete.
\end{remark}

\begin{dfn}\label{dfn:slice}
  Let $G$ be a Lie group acting continuously on a topological manifold $X$.
  A \emph{slice} through $p \in X$, is a closed embedded submanifold $S_p$ of $X$ containing $p$ such that:
  \begin{enumerate}[label=(\roman*)]
  \item\label{item:slice-i} For any $g \in G_p$, we have $g \cdot S_p \subset S_p$.
  \item\label{item:slice-ii} If $g \in G$ is such that $g\cdot S_p \cap S_p\neq \emptyset$, then $g\in G_p$.
  \item\label{item:slice-iii} There exists an open neighborhood $U$ of the identity right-coset in $G/G_p$, and a cross-section $\chi\colon U\to G$, such that the map $F \colon U\times S_p\to X$ given by 
    \[
      F(u,s) = \mu(\chi(u),s),
    \]
    is a homeomorphism onto an open neighborhood of $p$ in $X$.
  \end{enumerate}
\end{dfn}

\begin{remark}
  This term can be naively generalized in several directions.
  Let $\mathcal C$ be a subcategory of manifold objects in a category of topological spaces with finite products and group objects.
  Let $X$ be an object and $G$ be a group object in $\mathcal C$ acting on $X$ internally, i.e.\ the action $G \times X \to X$ is a morphism in $\mathcal C$.
  For $p \in X$, a \emph{slice} is a closed embedding $i \colon S_p \hookrightarrow X$ with $p \in \operatorname{Im} i$ of an object $S_p$ in $\mathcal C$ satisfying \ref{item:slice-i} -- \ref{item:slice-iii} above.
  If $F$ in \ref{item:slice-iii} can be strengthened to be a morphism in $\mathcal C$, we call it an \emph{internal slice}.

  The main theorem states the existence of a slice for the action of $\Diff(M)$ on $\mathcal R(M)$ considered in $\mathcal C$, the category of Fréchet manifolds.
\end{remark}

We will give a proof of the existence of a slice for the case of finite-dimensional Lie groups and smooth manifolds, which we will later retrace in the infinite-dimensional setting.

\begin{thm}[Slice Theorem~3.49 in \cite{AlexandrinoBettiol}]\label{t: classical slice theorem}
  Consider a proper smooth action by a Lie group $G$ on $M$, a smooth manifold.
  For arbitrary $p\in M$ there exists a slice $S_p$ through $p$.  
\end{thm}

\begin{proof}
  First we consider the $G$-invariant Riemannian metric $\gamma$ given by Theorem~\ref{t: existence of G-invariant Riemannian metric}.
  From the invariance  of the metric $\gamma$, it follows that if $\lambda\colon I\to M$ is a geodesic then for any $g\in G$, the curve $g\cdot\lambda(t)$ is a geodesic. 
  Next we fix $p\in M$, and consider the normal space $\nu_p M$ to $G(p)$ at $p$.
  We can find an open set $B$ contained in the tangent bundle $TM$, over which the exponential map is a diffeomorphism onto its image.
  Set $B^\perp$ the intersection of $B$ with the normal bundle $\nu G(p)\to G(p)$.
  Set $\varepsilon>0$ small enough, and let $V$ consist of all normal vector $v\in B^\perp$, with norm less than $\varepsilon$.
  Take $V_p\subset V$ to be a open ball of radius $\varepsilon$, around the origin, contained in the intersection of $B$ with the normal space  $\nu_p M$ of $G(p)$ at $p$.
  We claim that the image of $V_p$ under $\expo_p$ is the desired slice through $p$.

  From the fact that the action maps geodesics to geodesics, points \ref{item:slice-i} and \ref{item:slice-ii} follow: Take $g\in G_p$ and $s =\expo_p(t_0 X)\in S_p$.
  Then for the curve $\lambda(t) = \expo_p(tX)$ we have that $g\cdot \lambda (t)$ is a geodesic.
  Since $\lambda(0)=p$, and $g\in G_p$, then $g\cdot \lambda (0)= p$.
  Furthermore from the fact that $X$ is normal to the orbit at $p$, and the $G$-invariance of $\gamma$, it follows that $g\cdot \gamma$ is normal to the orbit at $p$.
  Thus $g\cdot s=g\cdot\lambda(t_0)$ lies in $S_p$.
  Thus $G_p(S_p)\subset S_p$.

  Now consider $g$  an arbitrary element in $G$, such that for some $s\in S_p$, $g\cdot s\in S_p$.
  This means that $g	\cdot s = \tilde{s}$ for some $\tilde{s}\in S_p$.
  Assume that $g\cdot p\neq p$.
  Then we have one point in the image of $B$ under the exponential map $\expo$, with two different foot points: $g\cdot s$ has $g\cdot p$ as base point for $\expo$, and $\tilde{s}$ has $p$ as a base point for $\expo$.
  Since the exponential map is a diffeomorphism on $B$, this is a contradiction.
  Thus $g\cdot p = p$.

  To prove point \ref{item:slice-iii} we need to consider the right action of $G_p$ on $G$, given by the group multiplication.
  Since this action is free, 
  the fibration $\rho\colon G\to G/G_p$ is a (smooth) $G_p$-principal bundle (cf.\ \cite[Corollary~3.38]{AlexandrinoBettiol}).
  Thus there exists an open neighborhood $U$ of the  coset $eH\in G/H$, such that the bundle $\rho|_{\rho^{-1}(U)}\colon \rho^{-1}(U)\to U$ is trivial.
  i.e.\ there exists a cross-section $\chi\colon U\to G$ of $\rho$.
  By Theorem~\ref{t: orbit of proper action in homeo to G/Gp} there exists a homeomorphism $\varphi\colon G/G_p\to G(p)$.
  Thus we may assume that $U$ is such that $\tilde{U} = \varphi (U)\subset \expo(V)$.
  Consider $(g_1G_p,s_1), (g_2G_p,s_2)\in G/G_p\times S_p$ and assume that $F(g_1G_p,s_1) =F(g_2G_p,s_2)$.
  Then we have that 
  \[
    \chi(g_2G_p)^{-1}\chi(g_1G_p) \cdot s_1 = s_2\in S_p.
  \]
  By point \ref{item:slice-ii} we have that $\chi(g_2G_p)^{-1}\chi(g_1G_p)\in G_p$.
  Thus $g_1G_p = g_2G_p$, and from this it follows that $s_1=s_2$.
  Therefore $F$ is injective. 


  Given $q\in \expo(V)$ there exists a unique vector $v\in V$, normal to the orbit such that $\expo(v)$.
  Furthermore we can assume that for the projection map $\pi\colon V\to G(p)$, we have that $x=\pi(v) \in \tilde{U}\subset G(p)$ (i.e we take $q\in \pi^{-1}(\tilde{U})$).
  Then, there exists $g\in G$ such that $x = g\cdot p$.
  Thus for the homeomorphism  $\tilde{\mu}_p$ given by Theorem~\ref{t: orbit of proper action in homeo to G/Gp}, we have that $\tilde{\mu}_p(x) = gG_p \in U$.
  Observe that 
  \[
    \chi(gH)\cdot p = \mu_p(\chi(gH)) = \tilde{\mu}_p\circ \rho(\chi(gH) ) =\tilde{\mu}_p(gH) = g\cdot p = x.
  \]
  Therefore $\chi(gH)^{-1}\cdot x = p$.
  We remark that, since $q$ lies in a geodesic normal to $G(p)$, and the Riemannian metric $\gamma$ on $M$ is $G$-invariant, then $\chi(gH)^{-1}\cdot q$ lies in a geodesic normal to $G(p)$, which goes through $p$.
  Thus $ \chi(gH)^{-1}\cdot q$ lies on the slice $S_p$.
  From this discussion it follows that $(gG_p,\chi(gH)^{-1}\cdot q)\in G/G_p\times S_p$ and 
  \[
    F(gG_p,\chi(gH)^{-1}\cdot q) = q.
  \] 
  In other words for $q\in \expo(\pi^{-1}(\tilde{U}))$,  the inverse of $F$ is given by 
  \[
    F^{-1}(q) = \Big(\tilde{\mu}_p^{-1}\big(\pi(\expo^{-1}(q))\big),\chi\big(\tilde{\mu}_p^{-1}(\pi(\expo^{-1}(q)))\big)\cdot q\Big)
  \]
  Therefore $F^{-1}$ is a continuous function since it is given by continuous functions, and thus $F$ is a homeomorphism. 
\end{proof}

The following result depends only on the existence of slices for every point in the space acted on.

\begin{thm}\label{t: tubular neighborhood G-invariant metric space}
  Let $G$ be a Lie group, $X$ a topological manifold.
  Assume $\mu \colon G \times X \to X$ is a proper action, such that for any point $p \in X$, a closed slice $S_p$ exists, satisfying \ref{item:slice-i}, \ref{item:slice-ii} and \ref{item:slice-iii}, and the isotropy group $G_p$ is compact.
  Then a closed neighborhood of the orbit $G(p)$ is homeomorphic to 
  \[
    G \times_{G_p} S_p
  \]
\end{thm}

\begin{proof}
  We set $\mbox{Tub}(G(p)) = \mu(G\times S_p)$.
  Since the action of $G_p$ leaves $S_p$ invariant, then we consider the following action of $G_p$ on $G\times S_\gamma$.
  For $h\in G_p$ and $(g,s)\in S_p$, we set
  \[
    h\star (g,s) = (gh,\mu(h^{-1},s) )
  \]
  We denote the orbit space of this action by $G\times_{G_p} S_p$, and we consider $\pi\colon G\times S_p\to G\times_{G_p} S_p$ the orbit projection map.
  We define  $\phi\colon G\times S_p\to \mbox{Tub}(G(p))$ by $\phi(g,s) =\mu(g,s)$.
  Thus by definition the map $\phi$ is continuous and surjective.
  For $(g,s) $ and $(gh,\mu(h^{-1},s))$, we have that $\phi(g,s)=\phi(gh,\mu(h^{-1},s))$.
  From the fact that  $\pi$ is a quotient map, it follows that there exists a continuous map $\psi\colon G\times_{G_p} S_p\to \mbox{Tub}(G(p))$ making the following diagram commute
  \begin{center}
    \begin{tikzcd}
      G\times S_p \arrow{d}[swap]{\pi}\arrow{dr}{\phi} & \\
      G\times_{G_p} S_p \arrow{r}[swap]{\psi} & \mbox{Tub}(G(p))
    \end{tikzcd}
  \end{center}
  We will show that the map $\phi$ is an open map.
  Let $U\subset G/H$ an open neighborhood of the identity coset, given by point \ref{item:slice-iii} of \ref{dfn:slice}.
  Consider the quotient map $p\colon G\to G/H$.
  Then $p^{-1}(U)\times S_p\subset G\times S_p$ is open. 
  Consider the map $F\colon U\times S_p\to X$ given by point \ref{item:slice-iii}.
  Set $N= F(U\times S_p)\subset X$, which is open.
  Since  we have $\chi(U)\subset p^{-1}(U)$, for the cross-section  $\chi\colon U\to G$, then $N\subset \phi(p^{-1}(U)\times S_p)$.
  Consider $(g,s)\in G\times S_p$.
  Since the action of $G_p$ on the fibers of $p\colon G\to G/H$ is transitive, then there exists $h\in G_p$ such that
  \[
    g= \chi\big(gH\big)h
  \]
  Thus $\phi(g,s) = \phi(\chi\big(gH\big)h,s) = \mu(\chi\big(gH\big)h,s) = \mu(\chi\big(gH\big),\mu(h,s)) $.
  Since $\mu(h,s)$ is an element of $S_p$ by property $(i)$ of the Slice, then $\phi(g,s)\in N$.
  Thus the image under $\phi$ of an open set in $G\times S_p$ is open in $\mbox{Tub}(G(p))$.
  Therefore $\phi$ is a quotient map, and $\psi$ is a homeomorphism.
\end{proof}


\subsection{Fr\'{e}chet spaces}\label{ss: Frechet spaces}
Let $F$ be a real vector space.
A \emph{seminorm} on $F$ is a real valued function $\norm{\cdot} \colon F\to \R$ such that:
\begin{enumerate}[label=(\roman*)]
\item $\norm{f} \geqslant 0$ for any vector $f\in F$;
\item $\norm{f + g} \leqslant \norm{f}+\norm{g}$ for all vectors $f, g\in F$;
\item $\norm{\lambda f} = \abs{\lambda} \norm{f}$ for any vector $f\in F$, and any $\lambda\in \R$.
\end{enumerate}
A family of seminorms $\{\norm{\cdot}_\alpha\mid \alpha \in \Lambda\}$ defines a unique topology on $F$ such that a sequence $\{f_k\}$ converges to $f$ if and only if for all $\alpha$ in $\Lambda$, the limit of $\norm{f_k-f}_\alpha$ is $0$.
We call $(F,\{\norm{\blank}_\alpha\})$ with this induced topology a \emph{locally convex topological vector space}.
This topology is metrizable if and only if $\Lambda$ is countable as a set.
In this case we say that a sequence $\{f_k\}$ is Cauchy if for any $\varepsilon >0$ there exists an $N\in \N$, such that for $n,m \geqslant N$ and for all $\alpha\in \Lambda$ we have $\norm{f_n-f_m}< \varepsilon$.
We say that a metrizable locally convex topological vector space is \emph{complete} if for any Cauchy sequence $\{f_k\}$ there exist $f\in F$, such that $\{f_k\}$ converges to $f$ in $(F,\{\norm{\blank}_\alpha\})$.
A \emph{Fr\'{e}chet space} is a complete, metrizable, locally convex topological vector space.

We can define derivatives in Fr\'{e}chet spaces as follows.
Let $E$ and $F$ be  Fr\'{e}chet spaces, and $U \subset E$ an open subset.
Consider a continuous map $f \colon U \to F$.
We say that $f$ is \emph{differentiable at $x \in U$} in the direction of $v \in E$, if the \emph{differential}
\[
  Df(x)v := \lim_{t\to 0}\frac{f(x+tv)-f(x)}{t}.
\]
exists.
We say that $f$ is of class $C^1(U)$ if the limit exists for arbitrary pairs $(x,v)$ in $U\times E$, and the map $Df\colon U\times E\to F$ is continuous in both arguments.
By considering 
\[
  \lim_{s\to 0}\frac{Df(x+s h)v -Df(x)v}{s}, 
\]
we can define the second derivative $D^2f\colon U\times E\to E\to F$, of $f$.
\begin{remark}
  It does not make sense to consider the partial derivative of $Df$ with respect to $v$, since $Df(x)v$ is linear in $v$.
\end{remark}
By iterating this definition, we can define for $k\in \N$ maps of class $C^k$, and maps of class $C^\infty$ between Fr\'{e}chet spaces.

\begin{dfn}
  Given a Hausdorff topological space $\mathcal M$, we say that it is \emph{modeled on a Fr\'{e}chet space} $F$ in an analogous way to smooth finite-dimensional manifolds.
  Namely, $\mathcal M$ is modeled on $F$ if there exists an atlas $\{(U_i,\phi_i)\mid i \in \mathcal{I}\}$, where each $U_i\subset \mathcal M$ is open, and each $\phi_i\colon U_i\to E$ is a homeomorphism onto its image.
  Furthermore, if $U_{ij}= U_i\cap U_j\neq \emptyset$, the transition map $\phi_i\circ(\phi_j)^{-1}\colon \phi_j(U_{ij})\to \phi_i(U_{ij})$ is a $C^\infty$ map of Fr\'{e}chet spaces.
\end{dfn}

\begin{remark}
  Since Fr\'{e}chet spaces are not necessarily finite-dimensional, a Fr\'{e}chet manifold is not necessarily finite-dimensional.
  Furthermore, they differ from Banach manifolds in several aspects.
  Among the most crucial differences is the absence of an inverse function theorem for arbitrary Fr\'{e}chet manifolds (cf.\ \cite{Hamilton}).
\end{remark}

\subsection{Compact-open topology}\label{ss: Compact-open topology}
We fix two smooth manifolds $M$, and $N$.
We denote the set of all smooth maps between $M$ and $N$ by $C^\infty (M,N)$.
Let $k \in \N \cup \{\infty\}$.
There exists a topology on $C^\infty (M,N)$ with the following property: a sequence of maps in $C^\infty (M,N)$ converges if and only if the first $k$ derivatives converge uniformly on compact subsets of $M$.
This topology is called the \emph{compact-open $C^k$-topology}\footnote{Other common names in the case $k = \infty$ are \emph{weak topology}, e.g.\ \cite{Hirsch}, or \emph{topology of uniform convergence over compact subsets}, e.g.\ \cite{TuschmannWraith}.}.
Whereas, if $M$ is a compact manifold, we will also call it the \emph{$C^k$-topology}.
In general, the compact-open $C^{\infty}$-topology is metrizable and $C^\infty (M,N)$ (cf.\ \cite[Corollary 41.12]{KrieglMichor}).
This topology is relevant to us by the following result:
\begin{thm}
  If $M$ is compact, then for any smooth vector bundle $\pi \colon E \to M$ over $M$, the space of smooth cross-sections $\Gamma(E)$ with the $C^{\infty}$-topology induced from $C^{\infty}(M, E)$ is a Fr\'{e}chet manifold.
\end{thm}

\begin{proof}
  See \cite[Sec.~42]{KrieglMichor}.
\end{proof}

In particular, the \emph{space of riemannian metrics} $\mathcal R(M)$, which is an open subset in $\Gamma(S^2T^{\ast}M)$, is a Fréchet manifold.


A topological group $G$, which is also a  Fr\'{e}chet manifold, is called a \emph{Fr\'{e}chet Lie group} if the operations of multiplication and taking inverses are smooth in the  Fr\'{e}chet sense.

\begin{thm}[Theorem~43.1 in \cite{KrieglMichor}]\label{T: the Diffeomorphis group is open in C(M;M)}
For a compact smooth manifold $M$ the group $\Diff(M )$ of all smooth diffeomorphisms of $M$ is an open Fréchet submanifold of $C^\infty(M, M )$, composition and inversion are smooth.
\end{thm}

\begin{cor}\label{cor:diff-is-a-frechet-lie-group}
 For $M$ a closed smooth manifold, the group of smooth diffeomorphisms $\Diff(M)$ is a Fr\'{e}chet Lie group.
\end{cor}


We can define the Lie Algebra of a Fréchet Lie group as the tangent space to the identity element. The Lie bracket is defined via the adjoint representation. The Lie algebra can be identified with the set of left-invariant vector fields on the group (see \cite{KhesinWendt}). The Lie algebra of $\Diff(M)$ is identified in the following theorem.

\begin{thm}[Theorem~43.1 in \cite{KrieglMichor}]
  Let $M$ be a closed smooth manifold.
  The Lie algebra of the Fréchet Lie group $\Diff(M)$ is the Lie algebra $\Gamma(TM)$, of all smooth vector fields on $M$, equipped with the negative of the usual Lie bracket.
  Furthermore, the exponential map $\expo \colon \Gamma(TM) \to \Diff(M),\ V \mapsto \phi^V(1)$ is given by evaluating the flow $\phi^V$ of a vector field $V$ at time $1$.
\end{thm}

\begin{remark}\label{R: Diff(M) is metrizable via a complete-left-invariant metric}
Moreover, for $M$ compact, the Fréchet structure at the tangent space of $\mathrm{Id}_M\in \Diff(M)$ induces a complete left-invariant metric on $G$, inducing the topology (see Theorem at end of page 53 in \cite{Subramaniam}). 
\end{remark}


\subsection{Riemannian structure on the space of Riemannian metrics}\label{ss: Riemannian structure on Riem(M)}

We can define a Riemannian metric $\sigma$ on $\Riem(M)$ as follows.
Given a Riemannian metric $\gamma$ on $M$, we identify the tangent space of $\Riem(M)$ at $\gamma$ with the sections $\Gamma(S^2T^\ast M)$ (see for example \cite{KrieglMichor},\cite{Clarke}).
Consider $S$ and $T$, two symmetric $(0,2)$-tensors in $\Gamma(S^2T^\ast M)$.
We set
\[
  \sigma_\gamma(S,T) := \int_M \mathrm{tr}_\gamma (S,T) \dop\vol(\gamma)=\int_M \sum\gamma^{ij} S_{il}\gamma^{lm}T_{jm} \dop \vol(\gamma),
\]
where $\dop \vol(\gamma)$ denotes the riemannian volume density with respect to $\gamma$.

Surprisingly, the topology induced by the Riemannian metric $\sigma$ on $\Riem(M)$ is weaker than the topology induced by the Fr\'{e}chet atlas.
For this reason, we say that $\sigma$ is a \emph{weak} Riemannian structure.
One reason why we are not recovering the original topology on $\Riem(M)$ via $\sigma$ is that $\Riem(M)$ equipped with $\sigma$ is not complete in the sense that each tangent space (which can be identified with $\Gamma(S^2T^\ast M)$) is not complete with respect to the interior product $\sigma$. 
Its completion is a Sobolev space (cf.\ \cite{Ebin1968},\cite{Gil-MedranoMichor1991}).

We note that it is still possible to take directional derivatives of $\sigma$, and thus we can attempt to define the Levi-Civita connection of $\sigma$ via the Koszul formula:
\begin{align*}
  \sigma(\nabla_X Y,Z) = & X\sigma(Y,Z)+Y\sigma(Z,X)-Z\sigma(X,Y)\\
                    &-\sigma(X,[Y,Z])-\sigma(Y,[X,Z])+\sigma(Z,[X,Y]).
\end{align*}
The fact that the tangent spaces of $(\Riem(M),\sigma)$ fail to be complete with respect to $\sigma$, implies that $\nabla_X Y$ may, a priori, not exist.
We point out that in case it does exists, then the Koszul formula implies it is unique.
The possible non-existence of $\nabla_X Y$ has as a  consequence the fact that there may be directions in $\Riem(M)$ where we cannot define the exponential map of $\sigma$.

A possible work-around for this problem, is to proceed as in \cite{Ebin1968}, and  complete the tangent spaces of $(\Riem(M),\sigma)$ to obtain Sobolev spaces.

In \cite{Gil-MedranoMichor1991}, they show that the exponential map, $\expo\colon T\Riem(M)\to \Riem(M)$, of  the metric $\sigma$ is actually defined everywhere in $\Riem(M)$, and for any tangent direction.
Furthermore, we  have the following theorem.

\begin{thm}[Theorem~45.13 in \cite{KrieglMichor}]\label{t: Exponential map of Riem(M) is a local diffeomorphism}
  The mapping \linebreak$(\pi,\expo)\colon T\Riem(M)\to \Riem(M)\times\Riem(M)$ is a diffeomorphism from an open neighborhood of the zero-section in $T\Riem(M)$ onto an open neighborhood of the diagonal in $\Riem(M)\times\Riem(M)$.
\end{thm}

This theorem has the following consequence.

\begin{cor}
  Given $\gamma \in \Riem(M)$, there exists an open neighborhood of $\gamma$ in $\Riem(M)$ over which the exponential map is a diffeomorphism.
\end{cor}

\subsection{Action of the diffeomorphism group}\label{ss: Action of the diffeomorphism group}

There is a smooth right action of $\Diff(M)$ on the space of Riemannian metrics $\Riem(M)$ via pullbacks.
Since we are interested in a left action, and the inverse map is smooth in $\Diff(M)$,  we consider the left action defined as
\begin{align*}
  \mu\colon \Diff(M)\times \Riem(M) \to \Riem(M),
  \quad
  (\phi,\gamma) \mapsto (\phi^{-1})^\ast\gamma. 
\end{align*}

We point out that this action is linear with respect to the Fr\'{e}chet structure.
I.e.\ for a fixed $\phi \in \Diff(M)$, the map 
\[
  \mu_\phi \colon \Gamma(S^2T^\ast M) \supset \mathcal R(M) \to \Gamma(S^2T^\ast M),
  \quad
  S \mapsto \mu(\phi,S)
\]
is a linear map.
We also point out that for a fixed element $\gamma\in \Riem(M)$, the isotropy group $\Diff(M)_\gamma$ consists of all isometries of $(M,\gamma)$.

As stated before, for $M$ compact, the group of smooth diffeomorphisms $\Diff(M)$ is a Fréchet Lie group.
We recall that there is an identification of the Lie algebra of $\Diff(M)$ with the algebra of smooth vector fields $\Gamma(TM)$.
Fix $\gamma\in \Riem(M)$, and set
\[
  \mu^{\gamma} \colon \Diff (M) \to \Riem (M),
  \quad
  \phi \mapsto \mu(\phi,\gamma).
\]
Then we have the derivative $(\mu^{\gamma})_{\ast} \colon \Gamma(TM) \to T_\gamma\Riem(M)$.
Furthermore, for a vector field $X \in \Gamma(TM)$ denote the flow of $X$ at time $t$ by $\phi^X_t \colon M \to M$.
In these terms, the derivative is given by
\[
  \left(\mu^{\gamma}\right)_\ast (X)
  = \frac{d}{dt}(\phi^{X}_{-t})^{\ast}\gamma
  = \mathcal{L}_{-X}(\gamma).
\]
\begin{remark}
Recall 
 that the Riemannian metric $\gamma$ induces a Riemannian metric on the bundle of $1$-forms (see \cite{AndrewsHopper}), which for $\omega_1,\omega_2\in \Omega^1(M)$ we denote by $\gamma(\omega_1,\omega_2 )$, as well.

If $X_1$ and $X_2$ are the dual vector fields of the $1$-forms $\omega_1$ and $\omega_2$, respectively, we have 
\[
  \gamma(\omega_1,\omega_2 ) = \gamma(X_1,X_2).
\]
With this, for the Riemannian structure $\sigma$ introduced in Section~\ref{ss: Riemannian structure on Riem(M)} we have  the following (cf.\ \cite{Blair2000})
\[
  \sigma(\mathcal{L}_X(\gamma), S)_\gamma
  = -2 \int_M (\nabla_i S^{ij}) X_j\dop\vol(\gamma)
  = 2\gamma(\omega,\dive S).
\]
Here $\dive S$ is the divergence of $S$ with respect to $\gamma$.
Thus a symmetric $2$-tensor is perpendicular to the orbit of $\Diff(M)$ at $\gamma$, if and only if, the divergence of $S$ vanishes.
By the work of Berger and Ebin in \cite{BergerEbin1969}, the tangent space $\Gamma(S^2T^\ast M)$ of $\Riem (M)$ at $\gamma$ splits as 
\begin{equation}
  \label{eq:beger-ebin-splitting}
  \Gamma(S^2T^\ast M) \cong T_\gamma \Diff(M)(\gamma) \oplus \kernel(\dive)  
\end{equation}
From the discussion above, this decomposition is compatible with the inner product $\sigma_\gamma$ and we denote $\nu_{\gamma}\Diff(M)(\gamma) := \kernel(\dive)$.
\end{remark}

\begin{remark}
The metric $\sigma$ defined in section~\ref{ss: Riemannian structure on Riem(M)} is invariant under the action of $\Diff(M)$ on $\Gamma(S^2T^\ast M)$ via pull-backs.
This follows from the fact that the action of $\Diff(M)$ is linear (see \cite[p.~20]{Ebin1968}).
\end{remark}

%

Also it was proven by Subramaniam in his Ph.D.\ thesis (\cite{Subramaniam}), that the action of $\Diff(M)$ on $\Riem(M)$ is proper for a compact action.
For the sake of completeness we provide a sketch of the proof here.

\begin{thm}\label{t: Action of Diff on Riem is proper for a compact manifold}
  Let $M$ be a closed smooth manifold.
  Then the action of $\Diff(M)$ on $\Riem(M)$ is proper.
\end{thm}

\begin{proof}[{{Proof (see \cite[p.68ff]{Subramaniam})}}]

We recall from Remark~\ref{R: Diff(M) is metrizable via a complete-left-invariant metric} that  $\Diff(M)$ has a complete left-invariant metric.
Therefore, by  Proposition~\ref{prop: second characterization of proper action}, the condition of being proper is equivalent to the following statement: If we consider a sequence $\left(f_n\right)_{n\in\N}$ of diffeomorphisms of $M$ such that for some fixed Riemannian metric $\gamma$, the sequence of Riemannian metrics $\left(f_n^\ast\gamma\right)_{n\in \N}$ converges to $\gamma$ with respect to the $C^\infty$-topology in $\mathcal R(M)$, then there exists a subsequence $\left(f_{n_{k}} \right)_{k\in\N}$ converging with respect to the $C^\infty$-topology in $\Diff(M)$.
  
  We will prove that this statement holds in two steps.
  First we show that with respect to a cover there exists a subsequence converging on each neighborhood with respect to the $C^1$-topology.
  Second, we will show that the limit functions agree on the overlaps of the neighborhoods, and that the subsequence converges with respect to the $C^\infty$-topology.

  \begin{lemma}
    Let $\{U_{\lambda}\}_{1 \leq \lambda \leq N}$ be an open cover of $M$ by normal coordinates.
    Then there exists a subsequence $(f_{n_k})_{k\in\N}$ of $(f_n)_{n \in \N}$ converging over each $U_{\lambda}$ with respect to the $C^1$-topology.
  \end{lemma}

  Let $m$ denote the dimension of $M$.
  Consider an open cover of $M$ by normal coordinates $U_\lambda$ with $1 \leq \lambda \leq N$, each one centered at some point $p_\lambda\in M$.
  Choose an orthonormal basis $\{V_{\lambda}^1,\ldots,V_{\lambda}^m\}$ of $T_{p_\lambda}M$ with respect to the metric $\gamma$.
  Since $f^\ast_n\gamma$ converges uniformly to $\gamma$, given $K>1$ we have for all $\lambda$, $j$, and for sufficiently large $n$
  \[
    \left\lVert V_\lambda^j\right\rVert_{f_n^\ast\gamma} = \left\rVert (f_n)_\ast(V_\lambda^j)\right\rVert_\gamma< K
  \]
  This implies that  the vectors $V_\lambda^j$ are contained in the disk bundle $DM_{K}$, which has as fiber the disk of radius $K$ with respect to each metric $f^\ast_n\gamma$.
  We note that the total space $DM_K$ is compact since $M$ is compact.
  We now proceed to construct the subsequence $(f_{n_k})_{k\in \N}$.
  First we consider the sequences of points $(f_n(p_{\lambda}))_{n\in\N}$.
  We begin by considering the point $p_1$.
  From the compactness of $M$ we can find a convergent subsequence $(f_{n_{k_1}})_{k_1 \in \N}$ of functions of $(f_{n})_{n\in\N}$, such that when we evaluate them at $p_1$ we have a convergent sequence.
  Denote the limit point $q_1$.
  We then evaluate the functions in $f_{n_{k_1}}$ at $p_2$, and use the compactness of $M$ to obtain a new subsequence $(f_{n_{k_2}})$ of $(f_{n_{k_1}})$, with limit $q_2$.
  We do this finitely many times to obtain a subsequence $(f_{n_{k_N}})_{k_N\in\N}$ of $(f_n)_{n\in\N}$ such that for all $\lambda$ the sequence $(f_{n_{k_N}}(p_{\lambda}))_{k_N\in\N}$ converges to $q_{\lambda}$.

  We apply the same process to the subsequence $(f_{n_{k_N}})_{k_N\in\N}$, using the compactness of $DM_k$, but now with the all the bases $\{V_{\lambda}^1,\ldots,V_{\lambda}^m\}$ and indices $\lambda$, to obtain a subsequence $(f_{n_k})_{k\in \N}$ such that for each index $\lambda$, the sequence $(f_{n_k}(p_{\lambda}))$ converges to $q_{\lambda}$ in $M$, and for each index $1\leqslant j\leqslant m$, the sequence $((f_{n_k})_\ast(V_{\lambda}^j))$ converges to $W_{\lambda}^j\in T_{q_{\lambda}}M$ in $TM$.

  From now on we consider only the subsequence $(f_{n_k})_{k\in\N}$, and write it as $(f_n)_{n\in\N}$.

  So far we have considered arbitrary normal neighborhoods.
  We now discuss which should be the neighborhoods we should consider.
  Since we have that the $\gamma$-norm of $(f_n^\ast)_\ast(V_{\lambda}^j)$ is less than $K$, the limit vectors $W_{\lambda}^j$ have $\gamma$-norm less than $2K$.
  We will assume from now on that the normal neighborhoods $U_{\lambda}$ are normal balls of radius $\delta<2K$ centered at the points $p_{\lambda}$.

  Denote by $d_n$ the metric induced on $M$ by the Riemannian metric $f^\ast_n\gamma$, and by $d$ the metric induced by $\gamma$.
  By definition, since $f^\ast_n\gamma$ converges to $\gamma$ with the respect to the $C^\infty$-topology, we have that the metric functions $d_n$  converges pointwise to $d$.
  This implies that if $q$ is contained in the ball of radius $\delta$ centered at $p$ with respect to $d$, then for a given $\varepsilon$, and $n$ large enough, $q$ is contained in the $\delta + \varepsilon$ ball centered at $p$ with respect to $d_n$.
  Furthermore, from the energy functional the minimizing geodesic between $p$ and $q$ with respect to $f^\ast_n \gamma$ converges to the minimizing geodesic with respect to $\gamma$.
  This also implies that the injectivity radius of $f^\ast_n\gamma$ converges to the injectivity radius of $\gamma$.
  These observations imply that for a fixed $p$ and $\varepsilon$, if $q$ is contained in the normal ball of radius $\delta$ centered at $p$, then for large $n$ we have that $q$ is contained in the normal ball of radius $\delta + \varepsilon$ centered at $p$.

  Take $q \in U_{\lambda}$.
  In particular we can write 
  \[
    q = \expo_{p_{\lambda}}^\gamma\left(\sum_j a_{\lambda\, j}V_{\lambda}^j\right).
  \] 
  Then for large $k$ we have that $q$ is in the normal ball of radius $2\delta$.
  Thus we can write
  \[
    q = \expo_{p_{\lambda}}^\gamma\left(\sum_j a_{n\lambda\, j}V_{\lambda}^j\right).
  \]

  Recall that the minimizing geodesics $\alpha_n$ joining $p_{\lambda}$ to $q_{\lambda}$ with respect to $f^\ast_n \gamma$ converge to the minimizing geodesic $\alpha$ joining $q$ to $p$ with respect to $\gamma$.
  Since the vectors $\sum_j a_{n \lambda\, j}V_{\lambda}^j$ are the tangent vectors to the geodesics $\alpha_n$ at $p_{\lambda}$, and $\sum_{j}a_{\lambda\,j} V_{\lambda}^j$ is the tangent vector to $\alpha$ at $p_{\lambda}$, then for each $\lambda$ and $j$ we have that the coefficient $a_{n\lambda\,j}$ converges to $a_{\lambda\,j}$. 

  We now define the limit function $f_{\lambda} \colon U_{\lambda} \to M$.
  For $q \in U_{\lambda}$ with $q = \expo_{p_{\lambda}}^\gamma\left(\sum_{j}a_{\lambda\, j}V_{\lambda}^j\right)$ we set
  \[
    f_{\lambda}(q) = \expo_{q_{\lambda}}^\gamma\left(\sum_j a_{\lambda\, j}W_{\lambda}^j\right).
  \]
  Observe that this function is smooth over $U_{\lambda}$.
  Using the fact that  $f_n$ is an isometry from $(M,f^\ast_n\gamma)$ to $(M,\gamma)$, and that   $q = \expo_{p_{\lambda}}^{f^\ast_n\gamma}\left(\sum_{j}a_{n\lambda\, j}V_{\lambda}^j\right)$, we have:
  \[
    f_n(q) = (f_n\circ \expo_{p_{\lambda}}^{f^\ast_n\gamma})\left(\sum_{j}a_{n\lambda\, j}V_{\lambda}^j\right) = \expo_{f_n(p_{\lambda})}^\gamma\left(\sum_{j}a_{n\lambda\, j}(f_n)_\ast\left(V_{\lambda}^j\right)\right).
  \]
  Thus the sequence $(f_n(q))_{n\in\N}$ converges to $\expo_{q_{\lambda}}^\gamma\left(\sum_j a_{\lambda\, j}W_{\lambda}^j\right) = f_{\lambda}(q)$ over $U_{\lambda}$ with respect to the $C^1$-topology.

  \begin{lemma}
    The limit functions $f_{\lambda}$ and $f_{\mu}$ agree on the overlaps $U_{\lambda} \cap U_{\mu}$ for all $\lambda, \mu$ and thus define a diffeomorphism $f$ on $M$.
    Moreover, there exists a subsquence of $(f_n)_{n \in \N}$ converging to $f$ with respect to the $C^\infty$-topology.
  \end{lemma}


  Consider $q\in U_{\lambda}\cap U_j$.
  On one side we have that the sequence $(f_n(q))_{n_\in\N}$ converges to $f_{\lambda}(q)$ in $M$.
  On the other side we have  that the sequence $(f_n(q))_{n_\in\N}$ converges to $f_j(q)$ in $M$.
  Thus $f_{\lambda}(q) = f_j(q)$.
  Since the limits $f_{\lambda}$ agree on the overlaps of the open cover $\{U_{\lambda}\}$ of $M$, then there is a global well defined function $f\colon M\to M$.
  This function $f$ is onto, since it is an open map, of a compact space to a connected Hausdorff space.
  It is also injective: This is true by construction, for the restriction of $f$ to each open set $U_{\lambda}$.
  Consider $d_\gamma$ the metric on $M$ induced by $\gamma$.
  From the compactness of $M$, it follows that there exists $L>0$ such that if $d_\gamma(p,q)<L$, then $p$ and $q$ lie in some $U_{\lambda}$.
  Take $p\neq q$ in $M$ and assume that $d_\gamma(p,q) >L$.
  We now consider $d_{f_n^\ast\gamma}$.
  Since $f_n$ is an isometry between $(M,d_{f_n^\ast\gamma})$ and $(M,\gamma)$, we have that $d_{f_n^\ast\gamma}(p,q)>L$.
  Since we have that $(f^\ast_n\gamma)_{n\in\N}$ converges to $\gamma$ uniformly over $M$, $(f_n(p))_{n\in\N}$ and $(f_n(q))_{n\in\N}$ converge  to $f(p)$ and $f(q)$ respectively, then $d_\gamma(f(p),f(q))\geqslant L$.
  In particular $f(p)\neq f(q)$.
  Thus $f$ is injective.

  Since by construction $f$ restricted to each $U_{\lambda}$ is smooth, then $f$ is a diffeomorphism of $M$.
  We will show that the sequence $(f_n)_{n\in\N}$ converges to $f$ with respect to the $C^\infty$-topology.
  To do this, we prove by induction that the sequence $(f_n)_{n\in\N}$ converges to $f$ with respect to the $C^k$-topology, for all $k\in \N$.
  The basis of induction is $k=1$.
  The statement holds true by construction.
  Next we assume that the sequence  $(f_n)_{n\in\N}$ converges to $f$ with respect to the $C^k$-topology.
  This implies that the first partial derivatives of the functions $f_n$  and $f_n^{-1}$ converge to the first derivatives of $f$ and $f^{-1}$ with respect to the $C^{k-1}$-topology.
  Let $\Gamma_n$ and $\Gamma$ denote the Christoffel symbols of $f_n^\ast\gamma$ and $\gamma$ respectively.
  Since $f^\ast_n\gamma$ is converging uniformly to $\gamma$ with respect to the $C^\infty$-topology, then for all indices $r,s,t$ we have that ${\Gamma_{n\,}}^{r}_{st}$ converges to $\Gamma^r_{st}$ with respect to the $C^\infty$-topology.
  Since each $f_n$ is a diffeomorphism, from the transformation law under change of variable, 
  \[
    {\Gamma_{n\,}}_{k\ell}^{j} =\sum_{r,s,t} \frac{\partial(f_n^{-1})^j}{\partial x^r}\frac{\partial (f_n)^s}{\partial x^k}\frac{\partial(f_n)^t}{\partial x^\ell}\Gamma^r_{st}
    + \sum_r \frac{\partial (f_n^{-1})^j}{\partial x^r}\frac{\partial^2(f_n)^r}{\partial x^k \partial x^\ell},
  \]
  we see that the second partial derivatives of $f_n$ must converges to the second partial derivatives $f$ with respect to the $C^{k-1}$-topology.
  This implies that the sequence $(f_n)_{n\in\N}$ converges to the $f$ with respect to the $C^{k+1}$-topology.
\end{proof}

\section{Proofs of the \autoref{thm:slice-thm} and \autoref{thm:tubular-nbhd-thm}}\label{s:Proof of slice theorem}

We begin by noting that given a slice by the \autoref{thm:slice-thm} and properness of the $\Diff(M)$ action by Theorem~\ref{t: Action of Diff on Riem is proper for a compact manifold}, the proof of \autoref{thm:tubular-nbhd-thm} is entirely analogous to the proof of Theorem~\ref{t: tubular neighborhood G-invariant metric space} in the Fréchet setting.

%
%
%
%
%
%
Our starting point for the proof of the slice theorem is the  Riemannian metric $\sigma$ on $\Riem (M)$ presented in section~\ref{ss: Riemannian structure on Riem(M)}, which is 
defined in local coordinates as,
\[
  \sigma(\alpha,\beta)_\gamma = \int_M \mathrm{tr}\big(\gamma^{-1}\alpha\gamma^{-1}\beta \big)\dop\vol(\gamma).
\]

We fix a Riemannian metric $\gamma\in \Riem(M)$.
Using the metric $\sigma$, and the fact that the orbit $\Diff(M)(\gamma)$ is closed, we split the tangent space of $\Riem (M)$ at $\gamma$ into two (infinite-dimensional) tangent subspaces: one tangent to the orbit $\Diff(M)(\gamma)$, and the normal complement $\mathcal{V}_\gamma \Diff(M)(\gamma)$ with respect to $\sigma_\gamma$ (see \eqref{eq:beger-ebin-splitting}).

\begin{remark}
  Recall that the space normal to the orbit at $\gamma$ is given by the symmetric $2$-forms with $0$ divergence with respect to $\gamma$ (see \cite{Blair2000, BergerEbin1969}).
\end{remark}

We will now follow the proof of Theorem~\ref{t: classical slice theorem} (the classical Slice theorem) to show the existence of slices for the action of $\Diff(M)$ on $\Riem(M)$.

Thus we need to show that  for 
the exponential map of $\sigma$,  is invariant under the action of $\Diff(M)$. Recall by Theorem~\ref{t: Exponential map of Riem(M) is a local diffeomorphism} that this exponential map  is a local diffeomorphism onto some open subset $U\subset \Riem  (M)$ around $\gamma$.

\begin{lemma}\label{L: invariance of exponential map of G}
  The exponential map of $\sigma$ is invariant under the action of $\Diff (M)$. 
\end{lemma}
\begin{proof}
  From the Koszul formula we have uniqueness for the Levi-Civita connection of $\sigma$.
  Since $\sigma$ is a $\Diff(M)$-invariant Riemannian metric, then by the uniqueness, the Levi-Civita connection is $\Diff(M)$-invariant.
  This implies that the action of  $\Diff(M)$ respects  geodesics.
  This implies that the exponential map is  $\Diff (M)$ invariant.
  See \cite{FreedGroisser1989, KrieglMichor, Gil-MedranoMichor1991}.
\end{proof}

With this lemma we can proceed to present the proof of the main theorem.

\begin{proof}[Proof of the \autoref{thm:slice-thm}]
We observe that we may consider 
the open \linebreak neighborhood $V$ to be a small ball around the origin in $\mathcal{V}_\gamma\Diff (M)(\gamma)$.
We set $S_\gamma = \expo_\gamma (V)$. We claim this is the desired slice.

We prove that for $S_\gamma$ defined above, point (i)  holds.
First, since the metric $\sigma$ is $\Diff(M)$-invariant, and $\Iso (\gamma)$ is a closed Lie subgroup, then $\sigma$ is also $\Iso (\gamma)$-invariant.
Recall that  $\eta\in S_\gamma$ is, by definition, contained in a geodesic $\lambda$, containing $\gamma$, and which is normal to the orbit $\Diff(M)(\gamma)$.
For $f\in \Iso (\gamma)$, from Lemma~\ref{L: invariance of exponential map of G}, we have that $f\cdot\lambda$ is a geodesic through $f\cdot\gamma = \gamma$.
Furthermore, from the invariance of $\sigma$ under the $\Iso(\gamma)$-action, and the fact that the geodesic $\lambda$ is normal to the orbit, then $f\cdot\lambda$ is normal to the orbit.
Also from the invariance of $\sigma$, it follows that the distance of $f\cdot \eta$ to the orbit $\Diff(M)(\gamma)$ is the same as the distance of $\eta$ to  the orbit $\Diff(M)(\gamma)$.
Thus by the definition of $S_\gamma$, we have that $f\cdot \eta$ lies in $S_\gamma$.
I.e.\ $f\cdot S_\gamma \subset S_\gamma$.
This proves point (i).

Point (ii) follows from the fact, due to Theorem~\ref{t: Exponential map of Riem(M) is a local diffeomorphism}, that the normal exponential map on the neighborhood $V\subset \nu_\gamma \Diff(M)(\gamma)$ of a point in the zero section is injective, when $V$ is small enough.
We now state how to show point (ii). Take $f\in \Diff (M)$, and $\eta, \psi \in S_\gamma$ such that $f\cdot \eta = \psi$ (i.e.\ assume $f\cdot S_\gamma \cap S_\gamma\neq \emptyset$).
Furthermore assume that $\psi$ and $\eta$ are at a distance less than $\delta>0$ from $\gamma$.
Since the action of $\Diff(M)$ is by isometries with respect to $\sigma$, it follows that the distance between $f\cdot \gamma$ and  $f\cdot \eta = \psi$ is  the same distance as the one between $\gamma$ and $\eta$.
Thus by the triangle inequality we have that the distance between $\gamma$ and $f\cdot\gamma$ is less than $2\delta$.
By taking $\delta$ small enough this implies that $f\cdot\gamma$ and $\gamma$ are contained in $V$.
Since $\eta$ is in  $S_\gamma$, then by the invariance of the Riemannian metric $\sigma$, and the invariance of the exponential map of $\sigma$, $\psi$ is contained in $S_{f\cdot \gamma}$.
From the fact that $\psi$ is contained in $S_\gamma$, it follows that $\psi$ is in  the image of $\expo_\gamma$.
On the other hand, $\psi$ lies in $S_{f\cdot\gamma}$, thus in the image of $\expo_{f\cdot \gamma}$.
Since the exponential map is injective on $V$, we conclude that $\gamma = f\cdot \gamma$.

The proof of point (iii) follows verbatim as in the proof of Theorem~\ref{t: classical slice theorem}, using the fact that the map $\rho\colon \Riem (M) \to\Riem(M)/ \Diff(M) $ is a principal fiber bundle.
\end{proof}
%

\begin{remark}
  The condition of properness for the action of $\Diff(M)$ on $\Riem(M)$ is necessary for the proof of point (iii), as we need that the orbit $\Diff(M)(\gamma)$ is homeomorphic to the homogeneous space $\Diff(M)/\Iso(\gamma)$.
\end{remark}
\section{Consequences of the Slice Theorem}\label{s: Consequences of the Slice Theorem}

For a given compact manifold $M$, several interesting consequences follow from the \autoref{thm:slice-thm} for the action of the diffeomorphism group  on the\linebreak space of Riemannian metrics.

The first one is \autoref{thm:tubular-nbhd-thm}, giving a description of a neighborhood of an orbit of $\Diff(M)$ inside $\Riem(M)$ up to homeomorphism.
This result follows from theorem~\ref{t: tubular neighborhood G-invariant metric space}.

Another interesting consequence is the study of how the isometries groups of Riemannian metrics which are close to each other in $\Riem(M)$ are related.

\begin{proposition}[Theorem~8.1 in \cite{Ebin1968}]
  Let $\gamma\in \Riem (M)$ be an arbitrary fixed Riemannian metric on $M$.
  Then there exists an open neighborhood $V$ of the identity in  $\Diff (M)$, and an open neighborhood $N$ of $\gamma$ in $\Riem (M)$, such that for any $\sigma\in N$, there exists an $f$ in $V$ such that:
  \[
    f^{-1}\Iso (\sigma) f\subset \Iso (\gamma).
  \]
\end{proposition}

\begin{proof}
  Take $V'$ any open neighborhood of the identity in $\Diff (M)$.
  Consider $U$ the open neighborhood of the identity coset in $\Diff (M) / \Iso (\gamma)$ given by the slice theorem.
  We recall that there exists a continuous cross-section $\chi\colon U\to \Diff(M)$ of the projection map $\pi\colon \Diff(M)\to \Diff(M) / \Iso(\gamma)$.

  Set $U' = U\cap \pi(V)$, which is open in $\Diff(M) /\Iso (\gamma)$.
  Setting $V = \pi^{-1}(U')$ we get an open neighborhood of the identity in $\Diff(M)$.
  We point out, that $\chi(U')\subset V$.
  We consider the homeomorphism $F\colon U\times S_\gamma \to \Riem(M)$, and set $N = F(U'\times S_\gamma)$.

  For $\sigma\in N$, it follows from the definition of $F$, that 
  \[
    \sigma = \chi(u')s,
  \]
  for some $u'\in U'$, and $s\in S_\gamma$.
  This implies that $s = \chi(u')^{-1}\sigma$.
  We set $f = \chi(u')^{-1}$.

  Consider $h\in \Iso (\sigma)$ arbitrary.
  Then by our choices, $(f^{-1}\circ h\circ f)(s) = s$.
  From point (ii) of the Slice Theorem it follows that $f^{-1}\circ h\circ f \in \Iso(\gamma)$.
  Since $h$ is arbitrary, it follows that
  \[
    f^{-1} \Iso (\sigma) f \subset \Iso (\gamma).
  \]
\end{proof} 

Let $\mathcal R_{\text{triv}}(M)$ be the collection of all Riemannian metrics on $M$ with trivial isometry group.
The following Corollary about $\mathcal R_{\text{triv}}(M)$ follows from the previous Proposition.

\begin{cor}[Corollary~8.2 in \cite{Ebin1968}, Theorem~1 in \cite{Kim1987}]
  $\mathcal R_{\text{triv}}(M)$ is open in $\Riem (M)$.
\end{cor}
\begin{proof}
  Take $\gamma\in \mathcal R_{\text{triv}}(M)$.
  Then, by the previous proposition there exists an open neighborhood $N$ of $\gamma$ in $\Riem (M)$ and $V$ an open neighborhood of the identity in $\Diff (M)$, such that for any $\sigma$ in $N$, we have
  \[
    f^{-1}\Iso (\sigma) f \subset \Iso (\gamma) =\{Id_M\},
  \]
  for some $f\in V$.
  This implies that $\Iso (\sigma)$ is trivial.
  Since $\sigma\in N$ is arbitrary, we conclude that $N\subset \mathcal R_{\text{triv}}(M)$.
  Thus $\mathcal R_{\text{triv}}(M)$ is open in $\mathcal R(M)$.
\end{proof}

Furthermore, given $\gamma$ an arbitrary Riemannian metric on a compact $n$-manifold, we can deform it in such a way that the absolute value of the Ricci curvature increases only on a small neighborhood $U$ of $M$. Then by increasing the  curvature at $n$ points inside $U$, we get a deformation of $\gamma$ which is at one point very asymmetrical. I.e. we are able to deform $\gamma$ into a Riemannian metric in $\mathcal R_{\text{triv}}(M)$.
\begin{thm}[Proposition~8.3 in \cite{Ebin1968}]
  The open set $\mathcal R_{\text{triv}}(M)$ is dense in $\Riem (M)$.
\end{thm}

\begin{remark}
The proof of the previous Theorem does not use the Slice Theorem. It only depends on the topology we consider in $\Riem (M)$.
\end{remark}

This means that a generic metric has trivial isometry group.
This agrees with the notion of the principal stratum in the setting of smooth proper Lie group actions on smooth finite-dimensional manifolds.
We consider the projection map $p\colon \Riem(M)\to \Mod (M)$ from the space of Riemannian metrics to the moduli space.
Then by continuity, the set $p(\mathcal R_{\text{triv}}(M))$ is open dense in $\Mod (M)$.
Furthermore, from the invariance of the metric $\sigma$ with respect to $\Diff(M)$, we can define a Riemannian metric on $p(\mathcal R_{\text{triv}}(M))$.
Thus we obtain the following Corollary:

\begin{cor}
  For a compact manifold $M$, an open dense set of $\Mod(M)$ admits a Fréchet, and a  Riemannian structure.
\end{cor}

Furthermore, the slice theorem allows us to lift paths from the moduli space, to $\Riem(M)$.

\begin{proposition}\label{prop: lift of paths}
Consider a path $\lambda\colon [a,b]\to \Mod(M)$, then for $\gamma \in \Riem(M)$ there exists a path $\tilde{\lambda}\colon [a,b]\to \Riem(M)$ with $\widetilde{\lambda}(a)=\gamma$ and $\pi\circ\widetilde{\lambda} = \lambda$.
\end{proposition}

\begin{proof}
By \autoref{thm:tubular-nbhd-thm}, for any $t\in [a,b]$ there exists an open neighborhood of $\lambda(t)$ homeomorphic to $S_{\gamma_t}/\Iso(\gamma_t)$, where $\gamma_t$ is any metric such that $\pi(\gamma_t) = \lambda(t)$. Since $[a,b]$ is compact, there exist finitely many open neighborhoods $S_{\gamma_{t_i}}/\Iso(\gamma_{t_i})$ covering $\lambda[a,b]$.  We recall that $\Iso(\gamma_t)$ is a finite dimensional Lie group. Then from \cite[Chapter~II, Theorem~6.2]{Bredon}, there exists such a lift over each open cover.
\end{proof}

\bibliographystyle{amsalpha}
\nocite{DiezRudolph2019}
\bibliography{Bibliography}

\end{document}